\documentclass[psamsfonts]{amsart}
\usepackage[utf8]{inputenc}
\usepackage{amsfonts}
\usepackage{hyperref}
\hypersetup{
pdftitle={On the quantum differentiation of smooth real-valued functions},
pdfsubject={Mathematics, Quantum Algebra, Q-calculus, Time-scale calculus},
pdfauthor={Kolosov Petro},
pdfkeywords={Quantum calculus, Quantum algebra, Power quantum calculus, Quantum difference, q-derivative, Jackson derivative, q-calculus, q-difference, Time Scale Calculus, Series Expansion, Taylor's theorem, Taylor's formula, Taylor's series, Taylor's polynomial, Analytic function, Series representation, Derivative, Differential calculus, Difference Equations, Numerical Differentiation, Polynomial, Exponential function, Exponentiation, Binomial coefficient, Binomial theorem, Binomial expansion, Mathematics, Numerical analysis, Mathematical analysis, arXiv, Preprint}
}
\usepackage{amsmath}
\usepackage{xcolor}
\usepackage{amsthm}
\usepackage{pdflscape}
\usepackage{pgfplots}
\usepackage{mathrsfs}
\usepackage[backend=bibtex]{biblatex} %
\newtheorem{thm}{Theorem}[section]

\newtheorem{lem}[thm]{Lemma}

\theoremstyle{definition}

\theoremstyle{remark}
\newtheorem{rem}[thm]{Remark}

\makeatletter
\let\c@equation\c@thm
\makeatother
\numberwithin{equation}{section}

\title{Another Topological Proof of the Infinitude of Prime Numbers}

\author{Jhixon Macías}
\address{Jhixon Macías\newline
University of Puerto Rico at Mayaguez, Mayaguez, PR, USA\newline
United States of America}
\email{jhixon.macias@upr.edu}

\keywords{Primes, Topology, Greatest Common Divisor}

\subjclass[2020]{11A41; 54G05; 54H11}

\begin{document}

\begin{abstract}
We present a new topological proof of the infinitude of prime numbers with a new topology. Furthermore, in this topology, we characterize the infinitude of any non-empty subset of prime numbers.
\end{abstract}

\maketitle
In \cite{Jhixon2023}, we introduce a new topology $\tau$ on the set of positive integers $\mathbf{N}$ generated by the base \begin{equation*}
    \beta := \{\sigma_n: n\in\mathbf{N}\}, \text{ where } \sigma_n := \{m\in\mathbf{N}: \gcd(n,m)=1\}.
\end{equation*}
\textcolor{black}{Indeed,}
\begin{thm}[\cite{Jhixon2023}]
\textcolor{black}{$\beta$ is a base for some topology on $\mathbf{N}$.}
\end{thm}
\begin{proof}
\textcolor{black}{It is clear that $\displaystyle\bigcup_{n\in\mathbf{N}}\sigma_n=\sigma_1=\mathbf{N}$.} \textcolor{black}{On the other hand, note that for every $n, m\in\mathbf{N}$, we have $\sigma_{nm}=\sigma_n\cap\sigma_m$\footnote{Note that for all $x, n, m \in \mathbf{N}$, we have $\gcd(x,n)=\gcd(x,m)=1$ if and only if $\gcd(x,nm)=1$.}. Therefore $\beta$ is a base for some topology on $\mathbf{N}$.}
\end{proof} 


\begin{rem}
Note that 
\begin{equation*}
    \begin{split}
    \sigma_n:&= \left\{\begin{array}{lcc}
     \ \ \ \ \ \mathbf{N}& \text{if} & n=1  \\
     \displaystyle\bigcup_{\substack{1\leq m< n\\ \gcd(m,n)=1}}n\left(\mathbf{N}\cup \{0\}\right)+m & \text{if} & n>1 
\end{array}\right.\\
    \end{split}
\end{equation*}
since for every integer $x$, it holds that $\gcd(n,m)=\gcd(n,nx+m)$, see \cite[Theorem 1.9]{niven1991introduction}. It is easily deduced from here that $\sigma_n$ is infinite for every positive integer $n$. Furthermore, it is deduced that the topology $\tau$ is strictly coarser than Golomb's topology.
\end{rem}

The topological space $\mathbf{X}:=(\mathbf{N},\tau)$ does not satisfy the $\mathrm{T}_0$ axiom, is hyperconnected, and ultraconnected. Among other properties, one that will be useful is that for \textcolor{black}{each integer $n>1$}, we have:
\begin{equation}\label{e1}
    \mathbf{cl}_{\mathbf{X}}(\{n\})=\bigcap_{\substack{p|n}}\mathbf{M}_p
\end{equation}
Here, $\mathbf{cl}_{\mathbf{X}}(\{n\})$ is the closure in $\mathbf{X}$ of the singleton set $\{n\}$, $p$ is a prime number, and $\mathbf{M}_p$ is the set of all multiples of $p$.

\textcolor{black}{To prove that equation (\ref{e1}) holds, we require the following Lemma.}
\begin{lem}[\cite{Jhixon2023}]\label{lm1}
\textcolor{black}{If $p$ is a prime number, then $\mathbf{cl}_{\mathbf{X}}(\{p\})=\mathbf{M}_p$.}
\end{lem}
\begin{proof}
\textcolor{black}{Let $x\in\mathbf{cl}_{\mathbf{X}}(\{p\})$. Now, if $x\notin\mathbf{M}_p$, then $\gcd(x,p)=1$. Consequently, $x\in\sigma_p$, which implies that $p\in\sigma_p$, a contradiction. Therefore, $x\in\mathbf{M}_p$.} \textcolor{black}{On the other hand, suppose $x\in\mathbf{M}_p$. Then, $x=np$ for some positive integer $n$. Now, take $\sigma_k\in\beta$ such that $x\in\sigma_k$. This implies that $\gcd(np,k)=\gcd(x,k)=1$, so $\gcd(p,k)=1$, and thus, $p\in\sigma_k$. Hence, $x\in \mathbf{cl}_{\mathbf{X}}(\{p\})$.}
\end{proof}
\begin{thm}[\cite{Jhixon2023}]
\textcolor{black}{Equation (\ref{e1}) holds.}
\end{thm}
\begin{proof}
    
    \textcolor{black}{Let $n>1$ a integer. Take $x\in \displaystyle\bigcap_{\substack{p|n}}\mathbf{M}_p$. By Lemma \ref{lm1}, we have that $x\in\mathbf{cl}_{\mathbf{X}}(\{p\})$ for every $p$ such that $p|n$. Therefore, for every $\sigma_k\in\beta$ with $x\in\sigma_k$, it holds that $p\in\sigma_k$ for all $p$ such that $p|n$. This implies $\gcd(k,n)=1$, and thus, $n\in\sigma_k$. Consequently, $x\in \mathbf{cl}_{\mathbf{X}}(\{n\})$.} \textcolor{black}{On the other hand, suppose $x\in \mathbf{cl}_{\mathbf{X}}(\{n\})$. Consider $\sigma_k\in\beta$ such that $x\in\sigma_k$. Then, $n\in\sigma_k$, and therefore $\gcd(n,k)=1$. This implies that $\gcd(p,k)=1$ for every $p|n$, and so $p\in\sigma_k$ for all $p$ such that $p|n$. Thus, by Lemma \ref{lm1}, we conclude that $x\in \displaystyle\bigcap_{\substack{p|n}}\mathbf{M}_p$.}
\end{proof}
\begin{rem}
\textcolor{black}{If $n=1$, then $\mathbf{cl_X}(\{n\})=\mathbf{N}$. Moreover note that, for every positive integer $n>1$, we have $1\notin\mathbf{cl_X}(\{n\})$ since $1\in\sigma_n$.}
\end{rem}

The objective of this short note is to provide a new topological proof of the infinitude of prime numbers, distinct from the topological proofs presented by Fürstenberg \cite{furstenberg1955infinitude} and Golomb \cite{golomb1959connected}, which, in fact, are similar except for the topology they use.


Let $\mathbf{P}$ denote the set of prime numbers. Additionally, for any set $A$, $\# A$ denotes the cardinality of $A$, and $\aleph_0:=\#\mathbf{N}$. Consider the following Proposition.

\begin{thm}\label{th1}
$\#\mathbf{P}=\aleph_0$ if and only if $\mathbf{P}$ is dense in $\mathbf{X}$.
\end{thm}
\begin{proof}
Suppose there are infinitely many prime numbers. Then, for any positive integer $n>1$, we can choose a prime $p$ such that $p>n$, and consequently, $p\in\sigma_n$ since $\gcd(n,p)=1$. Therefore, $\mathbf{P}$ is dense in $\mathbf{X}$. On the other hand, assume that $\mathbf{P}$ is dense in $\mathbf{X}$. Let $\{p_1,p_2,\dots, p_k\}$ be a finite collection of prime numbers and consider the  non-empty basic element $\sigma_x$ where $x=p_1\cdot p_2\cdots p_k$. Note that none of the $p_i$ belong to $\sigma_x$, but since $\mathbf{P}$ is dense in $\mathbf{X}$, there must be another prime number $q$, different from each $p_i$, such that $q\in\sigma_x$. Consequently, there are infinitely many prime numbers.
\end{proof}
Theorem \ref{th1} indicates that we only need to prove the density of $\mathbf{P}$ in $\mathbf{X}$ to establish the infinitude of prime numbers. Precisely, that is what we will demonstrate.

\textcolor{black}{To achieve our goal, consider the set $\mathbf{N_1} := \mathbf{N}\setminus\{1\}$ and the subspace topology }
\begin{equation*}
    \textcolor{black}{\mathbf{\tau_{1}} := \{\mathbf{N_1} \cap \mathcal{O} : \mathcal{O} \in \tau\} \ \ \text{generated by the base} \ \ \mathbf{\beta_1} := \{\mathbf{N_1} \cap \sigma_n : \sigma_n \in \beta\}.}
\end{equation*}
\textcolor{black}{Also, consider the topological subspace $\mathbf{X_1} := (\mathbf{N_1}, \mathbf{\tau_1})$} and the following Lemma: 

\begin{lem}\label{lm2}
    $\mathbf{P}$ is dense in $\mathbf{X_1}$ if and only if $\mathbf{P}$ is dense in $\mathbf{X}$.
\end{lem}
\begin{proof}
It is clear that if $\mathbf{P}$ is dense in $\mathbf{X}$, then $\mathbf{P}$ is dense in $\mathbf{X_1}$. On the other hand, it is clear that $\mathbf{X_1}$ is dense in $\mathbf{X}$. So, by transitive property of density, if $\mathbf{P}$ is dense in $\mathbf{X_1}$, then $\mathbf{P}$ is dense in $\mathbf{X}$.
\end{proof}

Now, let's prove that $\mathbf{P}$ is dense in $\mathbf{X_1}$.

\begin{thm}\label{th2}
$\mathbf{P}$ is dense in $\mathbf{X_1}$.
\end{thm}
\begin{proof}
In any topological space, it holds that the union of closures of subsets of that space is contained in the closure of the union of those sets. Therefore,
\begin{equation}\label{e2}
\bigcup_{p\in\mathbf{P}}\mathbf{cl}_{\mathbf{X_1}}(\{p\})\subset\mathbf{cl}_{\mathbf{X_1}}\left(\bigcup_{p\in\mathbf{P}}\{p\}\right)=\mathbf{cl}_{\mathbf{X_1}}(\mathbf{P})\subset\mathbf{N_1}
\end{equation}

On the other hand, from equation (\ref{e1}), it follows that
\begin{equation}\label{e3}
\bigcup_{p\in\mathbf{P}}\mathbf{cl}_{\mathbf{X_1}}(\{p\})= \bigcup_{p\in\mathbf{P}}\left(\mathbf{cl}_{\mathbf{X}}(\{p\})\cap\mathbf{N_1}\right)=\bigcup_{p\in\mathbf{P}}\mathbf{cl}_{\mathbf{X}}(\{p\})=\bigcup_{p\in\mathbf{P}}\mathbf{M}_p
\end{equation}

However, using the fundamental theorem of arithmetic, it can be easily shown that
\begin{equation}\label{e4}
\bigcup_{p\in\mathbf{P}}\mathbf{M}_p=\mathbf{N_1}
\end{equation}

Then, by using equations (\ref{e2}), (\ref{e3}), and (\ref{e4}), we conclude that $\mathbf{cl}_{\mathbf{X_1}}(\mathbf{P})=\mathbf{N_1}$, i.e., $\mathbf{P}$ is dense in $\mathbf{X_1}$.
\end{proof}

From Theorem \ref{th1}, Lemma \ref{lm2} and Theorem \ref{th2}, we can deduce that

\begin{thm}\label{th3}
There are infinitely many prime numbers.
\end{thm}


There are many proofs of the infinitude of prime numbers, such as Goldbach's Proof \cite[p.3]{aigner1999proofs}, Elsholtz's Proof \cite{elsholtz2012prime}, Erdos's Proof \cite{erdHos1938reihe}, Euler's Proof \cite{euler1748introductio}, and more recent ones, see , \cite{northshield2017two}, \cite{goral2020p}, \cite{mehta2022short} and \cite{goral2023green}. \textcolor{black}{Moreover, more than 200 proofs of the infinitude
of primes can be found in \cite{mevstrovic2012euclid}}. However, Fürstenberg's and Golomb's proofs are the only known a priori topological proofs, which, in essence, as mentioned earlier, are based on the same idea, except for the topology used. Despite being able to present a topological proof using the same idea with the topology $\tau$ (left as an exercise to the reader), we present a completely different proof, not only because of the topology used but also due to the underlying idea—proving  that $\mathbf{P}$ is dense in $\mathbf{X}$. 

Finally, we want to leave the reader with the following interesting Theorem.

\begin{thm}\label{thf}
Let  $A\subset\mathbf{P}$ non-empty.  Then, $A$ is dense in $\mathbf{X}$, if and only if, $\# A=\aleph_0$.
\end{thm}
\begin{proof}
Replace $\mathbf{P}$ with $A$ in the proof of Theorem \ref{th1}.
\end{proof}
Theorem \ref{thf} implies a new relationship between number theory and topology, at least we hope so. Indeed, to answer questions such as: are there infinitely many even perfect numbers? or its equivalent, are there infinitely many Mersenne primes? it suffices to check the density of these sets on $\mathbf{X}$. Certainly, it may not be easy, but it is possible. The advantage of working with $\mathbf{X}$ is that this space is hyperconnected, so any subset is either dense or nowhere dense.

\newpage
\printbibliography

@article{aigner1999proofs,
  title={Proofs from the Book},
  author={Aigner, Martin and Ziegler, G{\"u}nter M},
  journal={Berlin. Germany},
  volume={1},
  year={1999},
  publisher={Springer}
}

@article{golomb1959connected,
  title={A connected topology for the integers},
  author={Golomb, Solomon W},
  journal={The American Mathematical Monthly},
  volume={66},
  number={8},
  pages={663--665},
  year={1959},
  publisher={Taylor \& Francis}
}

@article{elsholtz2012prime,
  title={Prime divisors of thin sequences},
  author={Elsholtz, Christian},
  journal={The American Mathematical Monthly},
  volume={119},
  number={4},
  pages={331--333},
  year={2012},
  publisher={Taylor \& Francis}
}

@article{erdHos1938reihe,
  title={{\"U}ber die Reihe $\sum 1/p$},
  author={Erd{\H{o}}s, P},
  journal={Mathematica, Zutphen B},
  volume={7},
  pages={1--2},
  year={1938}
}

@article{euler1748introductio,
  title={Introductio in analysin infinitorum...; tomus primus},
  author={Euler, Leonhard and others},
  year={1748},
  publisher={Lausannae: apud Marcum-Michaelem Bousquet \& Socies}
}

@article{furstenberg1955infinitude,
  title={On the infinitude of primes},
  author={Furstenberg, Harry},
  journal={Amer. Math. Monthly},
  volume={62},
  number={5},
  pages={353},
  year={1955}
}

@article{mevstrovic2012euclid,
  title={Euclid's theorem on the infinitude of primes: a historical survey of its proofs (300 BC--2017) and another new proof},
  author={Me{\v{s}}trovi{\'c}, Romeo},
  journal={arXiv preprint arXiv:1202.3670},
  year={2012}
}

@article{northshield2017two,
  title={Two short proofs of the infinitude of primes},
  author={Northshield, Sam},
  journal={The College Mathematics Journal},
  volume={48},
  number={3},
  pages={214--216},
  year={2017},
  publisher={Taylor \& Francis}
}

@article{mehta2022short,
  title={A Short Generalized Proof of Infinitude of Primes},
  author={Mehta, Jay},
  journal={The College Mathematics Journal},
  volume={53},
  number={1},
  pages={52--53},
  year={2022},
  publisher={Taylor \& Francis}
}

@article{Jhixon2023,
  author = {E. Aponte, J. Macías, L. Mejías and J. Vielma},
  title = {A strongly connected topology on the positive integers},
  year = {2023},
  note = {Submitted for publication},
}

@article{goral2020p,
  title={p-Adic Metrics and the Infinitude of Primes},
  author={G{\"o}ral, Haydar},
  journal={Mathematics Magazine},
  volume={93},
  number={1},
  pages={19--22},
  year={2020},
  publisher={Taylor \& Francis}
}

@article{goral2023green,
  title={The Green-Tao theorem and the infinitude of primes in domains},
  author={G{\"o}ral, Haydar and {\"O}zcan, Hikmet Burak and Sertba{\c{s}}, Do{\u{g}}a Can},
  journal={The American Mathematical Monthly},
  volume={130},
  number={2},
  pages={114--125},
  year={2023},
  publisher={Taylor \& Francis}
}

@book{niven1991introduction,
  title={An introduction to the theory of numbers},
  author={Niven, Ivan and Zuckerman, Herbert S and Montgomery, Hugh L},
  year={1991},
  publisher={John Wiley \& Sons}
}
\end{document}